\documentclass{amsart}
\usepackage[all]{xy}
\SelectTips{cm}{}
\usepackage{amsmath}
\usepackage{amssymb}
\usepackage{amscd}
\usepackage{amsthm}
\usepackage{amsfonts}
\usepackage{amsxtra}

\theoremstyle{plain}

\newtheorem{Thm}{Theorem}[section]
\newtheorem{Lem}[Thm]{Lemma}
\newtheorem{Pro}[Thm]{Proposition}
\newtheorem{Cor}[Thm]{Corollary}

\newtheorem{bigthm}{Theorem}

\theoremstyle{definition}
\newtheorem{Def}[Thm]{Definition}

\newtheorem{Rem}[Thm]{Remark}

\newtheorem*{Rem-intro}{Remark}

\newcommand{\Dirlim}{\varinjlim}
\newcommand{\Invlim}{\varprojlim}
\newcommand{\Holim}{{holim}}

\newcommand{\FFF}{{\mathbb{F}}}
\newcommand{\KF}{K}
\newcommand{\ZZ}{{\mathbb{Z}}}
\newcommand{\QQ}{{\mathbb{Q}}}

\newcommand{\CC}{{\mathbb{C}}}

\newcommand{\RR}{{\mathbb{R}}}
 \newcommand{\PP}{{\mathbb{P}}}

\newcommand{\GL}{{\mathrm{GL\,}}}
\newcommand{\GLo}{{\mathrm{GL_o\,}}}

 \newcommand{\DX}{{\mathcal{D}}X}
 \newcommand{\DD}{{\mathcal{D}}}
 \newcommand{\FF}{{\mathcal{F}}}

  \newcommand{\IX}{{\mathcal{I}}X}

\def\:{\colon\!}
\catcode`@=11
\@xp\let\csname subjclassname@1991\endcsname \subjclassname
\@namedef{subjclassname@2010}{%
  \textup{2010} Mathematics Subject Classification}
\catcode`@=12

\begin{document}


\title[Spaces of sections of  Banach algebra bundles]
{Spaces of sections of  Banach algebra bundles}
\author[Farjoun]{Emmanuel Dror Farjoun}
\address{Department of Mathematics,
     Hebrew University of Jerusalem,
     Jerusalem 91904, Israel}
\email{farjoun@huji.ac.il}
\author[Schochet]{Claude L.~Schochet}
\address{Department of Mathematics,
     Wayne State University,
     Detroit MI 48202}
\email{claude@math.wayne.edu}

\thanks{ }
\keywords{ general linear group of a Banach algebra, spectral sequences,   localization, Bott-stable, $K$-theory for Banach algebras, unstable $K$-theory}
\subjclass[2010]{46L80, 46L85, 46M20, 55Q52, 55R20, 55T25}
\begin{abstract}

Suppose that $B$ is a $G$-Banach algebra over $\FFF = \RR$ or $\CC$,  $X$ is a  finite dimensional compact metric space,
$\zeta : P \to X$ is a standard principal $G$-bundle,  and $A_\zeta = \Gamma (X, P \times _G B)$ is the  associated  algebra of sections.
 We produce  a spectral sequence
which  converges to $\pi _*(\GLo A_\zeta )  $ with
\[
E^2_{-p,q}  \cong \check{H}^p(X ; \pi _q(\GLo B) ).
\]
A related spectral sequence converging to
$\KF _{*+1}(A_\zeta )$ (the real or complex topological $K$-theory) allows us to conclude  that if $B$ is Bott-stable, (i.e., if
$ \pi_*(\GLo B) \to \KF_{*+1}(B)$ is an isomorphism for all $*>0$) then so is $A_\zeta$.
\end{abstract}
\maketitle
\tableofcontents

 \section{Introduction}

   Suppose that $X$ is a finite-dimensional compact metric space and  $ \zeta : P  \to X$ is a principal $G$ bundle for some
   topological group $G$ that acts on a Banach algebra $B$\footnote{We
     consider Banach algebras over     $\FFF = \RR$ or $\CC$ and use
     $\KF_*$ to denote real or complex topological $K$-theory, respectively.} by algebra automorphisms.   Let $E = P  \times _G B \to X$  be the associated fibre bundle.
Assume that the fibre bundle is {\emph{standard}}.\footnote{Being \emph{standard} is a technical condition which is automatic if $X$ has the homotopy type
of a finite CW-complex. See Definition 5.1.}
  Let
 \[
 A_\zeta = \Gamma (X, E)
 \]
denote  the set of continuous sections of the bundle. This has a natural structure of a   Banach algebra. If $B$ is unital then
$A_\zeta $ is also unital, with identity the canonical
 section that to each point $x \in X$ assigns the identity in $E_x$.

 We are interested in $\GL A_\zeta $, the group of invertible elements
 in
 $A_\zeta$. (If $A_\zeta $ is not unital then we understand this to mean the kernel of the natural map
 $\GL(A_\zeta ^+) \to \GL(\CC )$.)
This is a space of the homotopy type of a CW-complex,  second countable if $B$ is separable. It may have many
(homeomorphic) path components; let  $\GLo A_\zeta $
denote the path component of the identity.

Let $\PP $ denote a subring of the rational numbers. (We allow the cases $\PP = \ZZ$ and $\PP = \QQ$ as well as intermediate rings.) Cohomology is reduced \v Cech cohomology.   Our principal result is the following. It is proved in Sections 4 and 5.

 \begin{bigthm}\label{T:theoremA}
  Suppose that $X$ is a finite-dimensional compact metric space  and  that
   $\zeta : E =   P  \times _G B  \to X    $ is a standard fibre bundle of Banach algebras over $\FFF$. Let
 $A_\zeta $ denote the associated section algebra.
Then:
 \begin{enumerate}
\item There is a second quadrant spectral sequence
  converging to $\pi _*(\GLo A_\zeta )\otimes \PP  $
 with
 \[
 E^2_{-p,q}  \cong \check{H}^p(X ; \pi _q(\GLo B)\otimes \PP )
 \]
 and
 \[
 d^r : E^r_{-p,q} \longrightarrow E^r_{-p-r, q+r-1}  .
 \]
 \item If
 $X$ has dimension at most $n$, then $E^{n+1} = E^\infty $.
  \item The spectral sequence is natural with respect to pullback diagrams
  \[
  \begin{CD}
  f^*E @>>>   E  \\
  @VVV  @VVV    \\
   X' @>f>>   X
 \end{CD}
 \]
 and associated map $f^*:   A_\zeta \longrightarrow  A_{f^*\zeta } $.
 \item The spectral sequence is natural with respect to $G$-equivariant maps $\alpha : B \to B'$ of   Banach algebras.

 \end{enumerate}
\end{bigthm}

Note that the $E^2$-term of the spectral sequence is independent of the bundle $\zeta $. Generally, this spectral sequence has several  types of non-zero differentials, and some of
these definitely {\emph{do}} depend upon the bundle. There is one very special case in which the spectral sequence does collapse at $E^2$.

 \begin{Cor}\label{collapsecorollary}  (\cite{LPSS}, \cite {KSS}) With the above notation,  if $A_\zeta $ is a unital continuous-trace $C^*$-algebra over $\CC $ and $\PP = \QQ $ then $E^{2} = E^\infty $
and
\[
\pi _*(\GLo A_\zeta )\otimes \QQ  \cong  \check{H}^*(X ; \pi _*(\GLo B)\otimes \QQ ) .
\]
\end{Cor}

Note that the assumptions in Corollary \ref{collapsecorollary} imply that   the Dixmier-Douady
invariant in $\check{H}^3(X; \QQ )$ is zero. Thus this corollary does not contradict the non-collapse results of Atiyah and Segal \cite{AS}. See Remark \ref{differentials} for more discussion of differentials in the spectral sequence.

The Corollary was established for $X$ compact metric, $B = M_n(\CC ) $, $\PP = \QQ $ and $\zeta $ trivial, so
  that $A_\zeta = F(X, M_n(\CC ))$, by Lupton, Phillips, Schochet, and Smith
\cite{LPSS} and in general by Klein, Schochet, and Smith \cite {KSS}.  The idea of using a spectral sequence to compute
$\pi _*(F(X,Y))$  goes back to Federer \cite{F}, and  note particularly the work of S. Smith \cite{Smith}. A more abstract
study of the homotopy groups of sections was undertaken by Legrand \cite{Leg} and his spectral sequence would seem
to overlap with ours in special cases.

This spectral sequence is analogous to the Atiyah-Hirzebruch spectral sequence. In fact, if   the functor $\pi _*(\GLo(-))\otimes\QQ $
is replaced
by $\KF_*(-)$ then this has been studied by J. Rosenberg \cite{Rosenberg}   in the context of continuous-trace $C^*$-algebras.
Our Theorem B generalizes his result.

Note that in many cases of interest, for instance    $B = M_n(\CC ) $,     the  groups $\pi _*(\GLo B)$ are unknown, and so the
integral version of the spectral sequence cannot be used directly to compute $\pi _*(\GLo A_\zeta )$. However, frequently the groups
$\pi _*(\GLo B)\otimes\QQ$ {\emph{are}} known and hence the rational form of the spectral sequence will be more tractable.

Our approach to this problem has two steps. First,   assume that $X$ is a finite complex,
 replace the space E by an appropriate diagram of spaces, localize the diagram, and then filter it.
     This is a consequence, noted in \cite{F},
  of the basic  Fibre Lemma of Bousfield and Kan  \cite{BK}.
    Next,   assume that $X$ is a finite-dimensional compact
metric space, write it as an inverse limit of finite CW-complexes, and use our first step results together with limit techniques of \cite{KSS}  to complete the argument.

Using the Banach algebra version of Bott periodicity established by R. Wood \cite{Wood} and M. Karoubi \cite{Kar} and taking limits of spectral sequences, we derive the following.

\begin{bigthm}\label{T:theoremB} Suppose that $X$ is a finite-dimensional compact metric space, $B$ is a    Banach algebra over $\FFF$, and $\zeta : E \to X$ is a standard fibre bundle of $B$-algebras.
Then there is a second quadrant spectral sequence
 \[
 E^2_{-p,q}(A_{\zeta _\infty }     )   \cong \check{H}^p(X ; \KF _{q+1}(B)\otimes \PP )   \Longrightarrow    \KF _{*+1}(A_\zeta )\otimes \PP
 \]
which is the direct limit of the corresponding spectral sequences converging to
\[
 \pi _*(\GLo (A_{\zeta}\otimes M_n(\FFF ) )\otimes \PP  .
 \]
If $X$ has dimension at most $n$ then $E^{n+1} = E^\infty $.
\end{bigthm}

This result, established in Section 6, is due to J. Rosenberg \cite{Rosenberg} when $A_\zeta $ is a continuous-trace $C^*$-algebra over a finite complex $X$.

Next,  compare the spectral sequences of Theorems A and B.

\begin{Def}\label{R;bottstable}  Let $B$ be a    Banach algebra over $\FFF $. Let us call $B$ {\emph{Bott-stable}} if the natural map
 \[
 \pi_*(\GLo B) \to \KF_{*+1}(B)
 \]
   is an isomorphism for every $* > 0$, or, equivalently, if the natural map
   \[
   \GLo B \to \Dirlim \GLo (B\otimes M_n(\FFF ))
   \]
   is a homotopy equivalence.    Similarly, let us call $B$ {\emph{rationally  Bott-stable}} if the natural map
 \[
 \pi_*(\GLo B)\otimes \QQ  \to \KF_{*+1}(B)\otimes\QQ
 \]
   is an isomorphism.
\end{Def}
For example, the Calkin algebra $\mathcal{L/K}$ is Bott-stable.
The Bott Periodicity theorem for Banach algebras  may be stated as follows:
If the Banach algebra $B$  is  {\emph{stable}} in the sense that the inclusion of a rank one projection induces an isomorphism
$B \overset{\cong}\longrightarrow  B\otimes {\mathcal{K}}$
(where  $ {\mathcal{K}}$ denotes the compact operators), then $B$ is Bott-stable.

We introduce the term {\emph{Bott-stable}} with some
trepidation, but we have been unsuccessful in locating a prior use of this concept in the literature.

 Bott stability is a bit weaker than K. Thomsen's use \cite{Thomsen} of {\emph{$K$-stable}}. He shows that many
 important complex $C^*$-algebras, such as
 $\mathcal O_n\otimes B$   $(\mathcal O_n$  the Cuntz algebra, $B$ a $C^*$-algebra), infinite-dimensional simple AF algebras, properly infinite von Neumann algebras and the   corona algebra of $\sigma$-unital $C^*$-algebras are   $K$-stable and hence
Bott-stable.

The following theorem is proved at the end of Section 7.

 \begin{bigthm}\label{T:theoremC}
Suppose that $X$ is a finite-dimensional compact metric space, $B$ is a   Banach algebra over $\FFF $,
 and $\zeta : E \to X$ is a standard fibre bundle of $B$-algebras.
If $B$ is Bott-stable
then   $A_\zeta $ is Bott-stable.   If $B$ is rationally Bott-stable then   $A_\zeta $ is rationally Bott-stable.
\end{bigthm}

We expect to apply these results in several directions. For example, we hope to shed light on differentials in the twisted $K$-theory spectral sequence by
looking at them in the homotopy spectral sequence and then mapping over.

 The  authors  wish to thank John Klein, Chris  Phillips, and Sam Smith for continuing support
in this and related papers and especially Max Karoubi and Jim Stasheff
 for extensive and very useful comments on earlier
versions of this work. Our editor, Jon Rosenberg, observed that our results might generalize  from the complex to the real setting without any essential modification and we
have followed his advice.

\section{The Non-Unital Case}

We pause to prove a technical result that allows us to handle the non-unital Banach algebra case uniformly with the
unital case.

Suppose   that $B$ is a non-unital  Banach algebra over $\FFF $ with a $G$ action, $P  \to X$ is a principal $G$ bundle, and
$A_\zeta = \Gamma (X, P  \times _G B )$. There is a natural unitization  $B^+$ of $B$ and a canonical
short exact sequence
\[
0 \to B \longrightarrow B^+ \overset{p}{\longrightarrow} \FFF  \to 0
\]
with a natural splitting $\FFF  \to B^+$. The $G$-action extends from $B$ to $B^+$ canonically with the group fixing the identity. Hence
there is
  a split exact sequence of bundles
\[
\begin{CD}
P \times _G B @>>>    P \times _G B^+  @>p>>    P \times _G \FFF  \\
@VVV  @VVV  @VVV   \\
X @>1>>   X @>1>>   X.
\end{CD}
\]
Define
 \[
  A_{\zeta ^+ } = \Gamma (X, P \times _G B^+) .
\]
The induced action of $G$ on $\FFF  $ must be the trivial action, and so $P \times _G \FFF \cong X  \times \FFF  $ is a trivial
line bundle.  Thus $\Gamma (X, P \times _G \FFF  ) \cong C(X, \FFF ) $ and
the induced map
\[
p_*:  \GLo (  A_{\zeta ^+ } ) \longrightarrow \GLo (\Gamma (X, \FFF - 0 ))  \cong C(X, \FFF  - 0)
 \]
 splits canonically.
 Recall that  the general linear group of the non-unital Banach algebra $A_\zeta $ is defined by
\[
\GL(A_\zeta ) = Ker [\GL ( (A_\zeta )^+)  \longrightarrow \GL( (A_\zeta )^+/ A_\zeta ) \cong  \FFF  - {0} ]
\]

\begin{Pro} (N. C. Phillips \cite{Phillips})
With the above notation,
\begin{enumerate}
\item
There is an isomorphism of topological groups
\[
\GL(A_\zeta )  \cong Ker \big[    p_*:  \GL (  A_{\zeta ^+ } ) \longrightarrow   C(X, \FFF  - 0)  \big].
\]
\item There is an isomorphism of topological groups
\[
\GL(A_\zeta )  \cong   \Gamma (X, P \times _G \GL B ).
\]
\item
There is a canonically split short exact sequence
\[
0 \to \pi _*( \GL(A_\zeta ))  \longrightarrow \pi _*(  \GL (  A_{\zeta ^+ } ) ) \longrightarrow \pi _*( C(X,  \FFF - 0) ) \to 0.
\]
\end{enumerate}
\end{Pro}

\begin{proof}
(Phillips) If $D$ is an ideal in any unital Banach algebra $E$ then one can identify
the unitization $D^+$ with the set
\[
\{ d + tI  \in E  \,\,|\,\,  d \in D, t \in \FFF  \} .
\]
 This identification is isometric on $D$ and a homeomorphism on $D^+$.
Taking $D = A_\zeta $ and $E  =   A_{\zeta ^+} $  implies that
\[
  A_\zeta   \cong Ker \big[ A_{\zeta ^+}  \longrightarrow C(X, \FFF ) \big]
\]
 which yields the identification of $\GL (A_\zeta )$.
 \end{proof}

This theorem allows us to write
\[
\GL A_\zeta  \cong \Gamma (X, P \times _G \GL (B) )
\]
in the non-unital case (where it is not at all obvious) as well as in the unital case (where it is true essentially by definition) and
hence develop the two cases at once.\footnote{The referee notes that for any algebra the group $\GL(A) $ has a  description as
\[
\GL(A) = \{ a \in A \,\,|\,\,  \exists \, b \in A\,\, :\,\,   ab  =  ba = a + b \}
\]
with multiplication $g \circ h := g + h - gh$ and unit $0$.  This identification makes the identification of  $\GL A_\zeta $ obvious.}

For completeness, note that the groups
$ \pi _j (\GLo (A_\zeta )) $ and $  \pi _j(\GLo (A_{\zeta ^+})) $
contain essentially the same information, so that if one is known the other is determined, as follows.

The groups $\pi _*( C(X,  \CC - 0) )$ are well-known:

\begin{enumerate}
\item
$
\pi _0(C(X, \CC - 0)) = \pi _0( F(X, K(\ZZ, 1)) = \check{H}^1(X; \ZZ)
$
  essentially by definition;
\vglue .1in
\item
$
\pi _1(C(X, \CC - 0))   = \check{H}^0(X; \ZZ)
$
\hglue .5in
by the early result of R. Thom \cite{Thom},  and
\vglue .1in
\item
$
\pi _j(C(X, \CC - 0))   = 0  \qquad\qquad j > 1
$
  also by Thom's result.
    \end{enumerate}

  The real case is even simpler, since $\RR - 0 \simeq  {\pm 1}$. Thus
  \[
  C(X, \RR - 0 ) \cong  C(X, {\pm 1} )  = C (\pi _0(X), {\pm 1} )
  \]
  which is a zero - dimensional set, finite if $X$  has a finite number of path components.

\vglue .1in
To summarize:

\begin{Thm} Suppose that $X$ is a compact metric space, $\zeta : P \to X$ is a principal $G$-bundle, $B$ is
a non-unital Banach algebra with a $G$-action, and $A_\zeta $ is the associated Banach algebra of sections.
Then
\begin{enumerate}
\item There are  naturally split short exact sequences
\[
0 \to \pi _j (\GL (A_\zeta )) \longrightarrow  \pi _j(\GL (A_{\zeta ^+})) \longrightarrow \check{H}^{1-j}(X ; \ZZ ) \to 0
\]
for $j = 0, 1$.
\item For $j \geq 2$ (and $j \geq 1$ in the real case)  there is a natural isomorphism
\[
\pi _j (\GLo (A_\zeta )) \overset{\cong}{\longrightarrow}  \pi _j(\GLo (A_{\zeta ^+})) .
\]
\end{enumerate}
\end{Thm}

\section{Diagrams and Localization}

Let $X$ denote  a finite simplicial complex.
Denote by $\IX$ the small category whose set of objects is the set of simplices of $X$ and  a morphism $\sigma \to \tau$ is a face inclusion
of  two simplices  in the complex $X.$
 Thus there is at most one morphism between any two objects of $\IX.$
The category $\IX$ serves as an indexing category for a diagram of spaces $\DX,$ namely, a functor:  $\DX : \IX \to Top $  of topological spaces
that assigns to each  simplex  in the simplicial complex $X,$  the underlying space of $\sigma.$ So  $\DX$ is a diagram of  contractible spaces, each homeomorphic to some $n-$ simplex, indexed by $\IX$.

  Now given a fibration $F \to T \overset{p}\to X,$  let us denote by  $F_\sigma = p^{-1}(\sigma )$ the inverse images of a given simplex in the total space.
 Notice that for any inclusion of simplices $\sigma\subseteq \tau $ there is   an inclusion of inverse images $F_\sigma\subseteq F_\tau$ and these inclusions are coherent.  Thus  the fibration (in fact any map to X) yields a diagram $\tilde F$ over $\IX$  made up of the various spaces  $F_\sigma,$ where $\sigma$ runs over all the simplices of $X.$

We now present the space of sections  as a homotopy limit using the following standard equivalence:
\begin{Lem}\label{sectionaslimit} There is a natural equivalence
\[
\Gamma (X, T)  \simeq \Holim _{\sigma \in \IX} F_\sigma  .
\]
\end{Lem}

\begin{proof}

Here is an outline of  a proof:  Let $\DX: \IX\to Top$ be the diagram as above.  It is not hard to see  that
it is a free diagram and thus for any other diagram $W:\IX\to Top$ indexed by $\IX,$ the homotopy limit of $W$
is given by the space of "strictly coherent maps of diagrams"

$$holim_{\IX} W\cong Map_{\IX}(\DX, W).$$

In addition, the total space itself  may be written as a homotopy colimit of the diagram $(F_\sigma),$
   $\,\ T\cong hocolim_{\IX}F_\sigma $ and of course  for the identity $X\to X$ we get $X\cong hocolim_{\IX}\DX.$


 We now notice that any section $X\to T$ gives a map of diagrams $\DX\to \tilde T,$ since any simplex $\sigma$ in $X$ is mapped to the space $F\sigma$ over it, an the various maps are coherent with respect to the inclusion of simplices.
Moreover, every such map of diagrams gives, by gluing, a map $X\to T $ which is a section.
Namely, given a map between the two diagrams of spaces $f: \DX\to \tilde T$ then $hocolim_{\DX} f$ is the desired section by the comment above, since in that case the homotopy colimit coincides with the strict direct limit, since both diagrams are free over $\IX$.


Thus the desired result follows  $$holim_{\IX}{F_\sigma}\cong holim_{\IX} \tilde T\cong Map_{\IX}(\DX, T)\cong \Gamma(X,T). $$

\end{proof}

Our interest lies in  $\GLo A_\zeta $ and its localizations.
Given $A_\zeta = \Gamma (X, P \times _G B)$, consider the fibre bundle $ P \times _G \GLo B \longrightarrow X $.
 There is a  natural comparison map:
\[
c:\quad  (\GLo A_\zeta )_\PP  = {\Gamma (X, P \times _G \GLo B )} _{o\,\PP} \simeq \big[\Holim _{\DX} F_\sigma \big]_\PP  \longrightarrow \Holim _{\DX}\left( \big[F_\sigma \big]_\PP\right)
\]
where we used  Lemma \ref{sectionaslimit} to produce  the left maps and
where the last map on the right is naturally associated to the localization. Notice that on the left hand side one first takes localization and then homotopy limit.

In order to facilitate the formulation of the following   formula for the space of sections,
let  us denote by  $\Holim _{\DX} \big[F_\sigma \big]_{\PP, o}$  the component of  the   localization of the homotopy limit that
 is hit via the above comparison map $c$ by the component of the identity in the  localization of space of sections.
This is a shorthand for the cumbersome $\left(  \Holim _{\DX}\left( \big[F_\sigma \big]_\PP\right)  \right)_o.$

\begin{Thm}

 The map of connected spaces
\[
c:\quad  (\GLo A_\zeta )_\PP \longrightarrow    \Holim _{\DX} \big[F_\sigma \big]_{\PP,o}
\]
is a homotopy equivalence.
\begin{proof}  Recall that  for any connected  pointed space, if the action of its fundamental group on itself and all higher homotopy
groups is trivial then  the space is  certainly {\emph{nilpotent}}. Thus  the underlying space of any  connected topological group is nilpotent. Thus the diagram $F_\sigma$ in our case is in fact  a diagram of nilpotent spaces,
since the fibres are all equivalent  as spaces to the identity component $\GLo$  of our group. Now we use a theorem
about the localization of homotopy limits over certain diagrams  of nilpotent spaces over
finite dimensional indexing categories.  Here we use the assumption that $X$ is a finite simplicial complex.
By the main theorem in (\cite{F} Theorem 2.2 and Corollary. 2.3),  the comparison
 map has a  homotopically discrete fibre  (possibly empty!) over each component of the range, so that
  the map $c$ is,   up to homotopy, a covering projection.  Notice that  localization  in the range of the map $c,$ acts  one component at a time. Thus  the restriction of $c$ to each component of its domain is
a homotopy equivalence, in particular when one restricts to the component of the identity.

\end{proof}
\end{Thm}

\begin{Cor} There is a natural isomorphism of homotopy groups of the following  connected spaces, taken with corresponding  base points:
\[
\pi _j(\GLo A_\zeta )\otimes\PP  \cong \pi _j  \big( \Holim _{\DX} \big[F_\sigma \big]_{\PP,o })\big)
\]
for all $j \geq 0$.
\end{Cor}


\section{The Spectral Sequence for Finite Complexes}

In this section we construct the spectral sequence that is the main object of our work. Use the notation of
the previous section and  continue to assume that the base space $X$
is a finite simplicial complex.
The simplicial complex $X$ is filtered naturally by its skeleta and this induces a  filtration $\DD _pX $ of the diagram $\DX$.

 The spectral sequence is a special case of the Bousfield-Kan spectral sequence
for homotopy limits \cite{BK}.     Here we give a presentation of it for the space of sections over a finite simplicial complex.

Define
\[
\FF _pX  =   \Holim _{\DD _pX} \big[F_\sigma \big]_{\PP,o }   .
\]

Form the associated
exact couple.
The appropriate grading turns out to be
\[
D^1 _{-p,q}  \cong \pi _{q-p}(\FF _pX )\otimes \PP
\]
and
\[
 E^1 _{-p,q} = \pi _{q-p} ( (\FF _{p-1}X)/(\FF _{p}X)) \otimes \PP
 \]
 with differential
 \[
 d^1 : E^1_{-p,q} \longrightarrow E^1_{-p-1,q}.
 \]
 We may identify the $E^1$ term by noting that
 \[
 E^1_{-p,q} \cong \pi _{q-p} (\Gamma (\vee _\alpha  S^p, (\GLo A_\zeta )|_{S^p} ) \otimes \PP
  \cong \pi _{q-p} (F_*(\vee _\alpha  S^p , \GLo B ))\otimes \PP  \cong
 \]
 \[
 \cong \oplus _\alpha \pi _{q-p}(\Omega ^p \GLo B)\otimes \PP  \cong \oplus _\alpha \pi _q(\GLo B) \otimes \PP \cong
  C^p(X ; \pi _q(\GLo B)\otimes \PP).
  \]
  so that
  \[
    E^1_{-p,q} \cong    C^p(X ; \pi _q(\GLo B)\otimes \PP),
  \]
  the cellular cochains of $X$ with coefficients in $\pi _q(\GLo B)\otimes \PP$.
  The $d^1$ differential is  the usual cellular differential  and so
   \[
  E^2 _{-p,q} \cong \check{H}^p(X ; \pi _q(\GLo B)\otimes \PP).
  \]
  The filtration is finite since $X$ is a finite complex, and so the resulting spectral sequence  converges to $\pi_* (\GLo A)\otimes \PP $.
   This proves Theorem A   for the case $X$ a finite simplicial complex.

This construction has two types of naturality associated with it.

\begin{Pro}\label{Cor:naturality}
\begin{enumerate}
\item
The exact couple and resulting spectral sequence are natural with respect to pullbacks
\[
\begin{CD}
f^*(E)  @>>> E  \\
@VVV @VVV  \\
X' @>f>>  X
\end{CD}
\]
where $f: X' \to X $ is a simplicial map of finite simplicial complexes and $f^*(E) \to X'$ is the pullback bundle.
\item The spectral sequence is natural with respect to bundle maps
\[
\begin{CD}
E @>>>   E'  \\
@VVV  @VVV  \\
X @>1>>   X
\end{CD}
\]
induced by a $G$-equivariant  map $B \to B'$ on the fibres.
\end{enumerate}
\end{Pro}

\begin{proof} Both of these assertions are immediate from the construction.
\end{proof}

\section{The Spectral Sequence for Compact Spaces}

We move now to the general case, with   finite-dimensional compact metric space  $X$ and principal $G$-bundle
 $\zeta : P \to X$.  The group $G$ acts on the Banach algebra $B$ and may be regarded as a subgroup of $Aut(B)$, but this is a very large group
 in general.

 Principal $G$-bundles are classified by maps $X \to BG$ and so there is a pullback diagram
 \[
 \begin{CD}
 P \times _G B @>>>     EG \times _G B  \\
 @VVV   @VVV   \\
 X @>f>>  BG
 \end{CD}
 \]
 but it is critical at this point that the space $BG$
be of the homotopy type (and not just the weak homotopy type) of a CW complex.

\begin{Def} The fibre bundle $  E = P \times _G B \to X$ is {\emph{standard}} if either $X$ is of the homotopy type
of a  finite complex or (if not) then
 the structural group of the bundle reduces to some subgroup $G \subset Aut(B) $ which is of the
 homotopy type of a $CW$-complex, or equivalently, that
  the classifying space $BG$  of the bundle  has the
   homotopy  type
of a CW complex.\footnote{
We note that non-standard bundles do exist. Take some topological group $G$ that is not of the homotopy type of a $CW$-complex. Then $BG$ also has this property
and hence the universal principal $G$-bundle is not standard.}
\end{Def}

The classical theorem of Eilenberg-Steenrod (\cite{ES} Theorem. X.10.1) gives us an inverse sequence of simplicial complexes $X_j$
and simplicial maps together with a homeomorphism $X \cong \Invlim X_j$. Let  $h_j : X \to X_j$ denote the resulting
structure maps.
Then, as shown in \cite{KSS},  for $j$ sufficiently large, there are maps\footnote{The result in \cite{KSS} depends
upon work of Eilenberg-Steenrod which requires that the target space $BG$ be of the homotopy type of a CW-complex. This is why  standard bundles are required.}
\[
f_j : X_j \to BG
\]
 and   a homotopy-commuting
  diagram of the form
\begin{equation}\label{pullback}
\begin{CD}
\GLo E  @>>>      \GLo E_j   @>>>   \GLo E_{j-1}   @>>>     EG \times _G \GLo B  \\
@VVV @VV{\zeta _j}V   @VVV  @VVV     \\
X @>{h_j}>>   X_j  @>>> X_{j-1}   @>{f_{j-1}}>>  BG.
\end{CD}
\end{equation}
where each square is a pullback.
Let $\zeta _j : \GLo E_j \to X $ denote the associated bundles. Then there is a natural isomorphism
\[
\Dirlim \pi _*(\Gamma (\zeta _j ) ) \overset{\cong}\longrightarrow \pi _*(\Gamma (\zeta )) \equiv \pi _*(\GLo A_\zeta ) .
\]
For each finite simplicial complex  $X_j$  there is an exact couple
\[
(E^1_{*,*} (j), D^1_{*,*}(j))
\]
constructed  in Section 4.  The maps $X_j \to X_{j-1} $  are simplicial and hence induce morphisms of exact couples by Corollary
 \ref{Cor:naturality}.
Define
\[
  D^1_{*,*} = \Dirlim _j                 D^1_{*,*}(j))
\]
and
\[
 E^1_{*,*}  = \Dirlim _j   (E^1_{*,*} (j))
\]
This forms an exact couple, since filtered direct limits of exact sequences are exact.
The resulting spectral sequence has
\[
E^2 \cong   \Dirlim E^2(j) \cong   \Dirlim  \check{H}^*(X_j ; \pi _*(\GLo B) \otimes \PP )
\cong
 \check{H}^*(X; \pi _*(\GLo B) \otimes \PP )
\]
and converges to
\[
\Dirlim \pi _*(\Gamma (\zeta _j ) ) \otimes \PP   \cong  \pi _*(\GLo A_\zeta )\otimes \PP .
\]
as required. (Note that \v{C}ech cohomology is required at this point since that is the cohomology  theory that respects inverse
limits of compact spaces.)

We turn next to the naturality claims.

Suppose given a pullback diagram
\[
  \begin{CD}
  f^*E @>>>   E  \\
  @VVV  @VVV    \\
   X' @>f>>   X.
 \end{CD}
 \]
Then  use the diagram (\ref{pullback}) to construct the spectral sequence for $A_\zeta = \Gamma (X, E) $.
Write $X = \Invlim X_j $ and $X' = \Invlim X_j'$.
We recall that  the natural maps
\[
[X' , X] \longrightarrow [X', X_j ]
\]
induce a bijection
\[
[X' , X] \overset{\cong}{\longrightarrow} \Invlim _j  [X', X_j ]
\]
by the universal property of inverse limits. Since $X_j$ is a finite complex,   appeal to \cite{KSS} Proposition 6.2
to yield  a natural bijection
\[
\Dirlim _i [X_i' , X_j ]  \overset{\cong}{\longrightarrow}    [X', X_j ].
\]
Combining these,  there is a bijection
\[
[X' , X] \overset{\cong}{\longrightarrow} \Invlim _j \,\, [X', X_j ]   \overset{\cong}{\longleftarrow}     \Invlim _j   \Dirlim _i \,\,[X_i' , X_j ]   .
 \]

Suppose    given a bundle $E \to X$ as usual, with associated diagram
\begin{equation}\label{pullback2}
\begin{CD}
\GLo E  @>>>      \GLo E_j   @>>>   \GLo E_{j-1}   @>>>     EG \times _G \GLo B  \\
@VVV @VV{\zeta _j}V   @VVV  @VVV     \\
X @>{h_j}>>   X_j  @>>> X_{j-1}   @>>>  BG
\end{CD}
\end{equation}
and with each square a pullback. Passing to cofinal subsequences and
renumbering,   expand  diagram (\ref{pullback2})  to
\begin{equation}\label{pullback3}
\xymatrix{
& \GLo E  \ar[r]\ar[d]& \GLo E_j  \ar[r]\ar[d]^{\zeta _j}& \GLo E_{j-1}
\ar[r]\ar[d] &   EG \times _G \GLo B  \ar[d]\\
& X \ar[r]^{h_j} &   X_j  \ar[r] & X_{j-1} \ar[r]&  BG  \\
X' \ar[ur]^f \ar[r]^{h'_j} &   X'_j  \ar[ur]^{f_j}\ar[r] & X'_{j-1}
\ar[ur]^{f_{j-1}}& &
}
\end{equation}
and define bundles $f^*\zeta _j$  over each $X_j'$ via  pullback diagrams
\[
\begin{CD}
f_j^*E_j @>>>   E_j  \\
@VVV   @VVV  \\
X_j'  @>>f_j>  X_j .
\end{CD}
\]
Proposition \ref{Cor:naturality} then  tells us that there is a morphism of spectral sequences
\[
E_{*,*}^*(A_{\zeta _j})  \longrightarrow E_{*,*}^*(A_{f^*\zeta _j})
\]
and taking direct limits of both sides there is  a morphism of spectral sequences
\[
E_{*,*}^*(A_{\zeta} ) \longrightarrow E_{*,*}^*(A_{f^*\zeta })
\]
as desired.   This completes the proof of the naturality of the spectral sequence with
respect to pullbacks of bundles.

 Suppose that $\alpha : B \to B'$ is a map of   Banach algebras over $\FFF $  that respects the
group action $G$ with corresponding bundles $E$ and $E'$ over $X$. Then there is
a natural bundle map
\[
\begin{CD}
EG \times _G \GLo B   @>>>     EG \times _G \GLo B'  \\
@VVV   @VVV   \\
X  @>{Id}>>   X
\end{CD}
\]
and a coherent three-dimensional diagram showing the naturality of diagram (\ref{pullback2}) with respect to $\alpha $.
It is then routine to show that the various subsequent constructions carry this naturality along, essentially regarding
$B$ as coefficients.

 This completes the proof of Theorem  \ref{T:theoremA}.


\section{Comparison Theorems}

The spectral sequence of Theorem  \ref{T:theoremA}
 is natural in several senses and these naturality results bring with them corresponding comparison theorems. These are routine consequences of the Comparison Theorem for spectral
sequences.\footnote{It turns out to be remarkably difficult to give proper credit for spectral sequence comparison theorems.
Mac Lane \cite{MacLane} Theorem 11.1 is a fine reference, but he gives credit to J.C. Moore \cite{Moore}  in the Cartan
Seminar 1954--5 (now available on line!) who gives credit to a paper of Eilenberg and Mac Lane. McCleary \cite{McCleary}  gives credit to
E. C. Zeeman, who credits Moore for the result that we use, more or less.  None of
these are exactly the theorem that we use here, but {\emph{hamaivin yavin}} (a phrase found in Kabbalistic literature meaning \lq\lq those who understand will understand.\rq\rq )   }

\begin{Thm} Suppose that $X$ is a compact space,  $B$ and $B'$ are Banach algebras over $\FFF $, and
that
 \[
 \zeta :  E \to X \qquad and\qquad  \zeta '  :  E' \to X
 \]
 are   two standard bundles of Banach algebras
  with fibres $B$ and $B'$ respectively. Suppose given   a map of Banach bundles $f:E \to E'$.
  Then there is an induced natural map of Theorem  \ref{T:theoremA} spectral sequences which induces the natural map on the
$E_2$ and $E_\infty $ levels.  If the induced map
\[
f_* : \pi  _*(\GLo B)\otimes\PP  \to \pi _*(\GLo B')\otimes\PP
\]
is an isomorphism, then the induced map
\[
  f_* : \pi _*(\GLo A_\zeta )\otimes\PP \to \pi _*(\GLo A_{\zeta '})\otimes\PP
\]
is an isomorphism.
\end{Thm}

\begin{proof}  The fact that the first induced map  is an isomorphism implies that the map of spectral sequences
is an  isomorphism at the $E^2$  level and \lq\lq the standard comparison theorem\rq\rq\, implies that the
map of spectral sequences
is an  isomorphism at the $E^\infty $ level, which implies that the second map is an isomorphism.
\end{proof}

\begin{Thm}
Suppose that $X$ and $X'$ are finite dimensional compact metric spaces.
 Suppose given  a Banach algebra $B$ over $\FFF$, a standard bundle
 $
 \zeta :  E' \to X'
 $
 with  fibre $B$
 and a continuous function $f: X \to X'$.
 Let $\zeta '  : E \to X$ denote the pullback bundle.
 This induces a natural map of spectral sequence with the natural maps
  on
$E_2$ and $E_\infty $.  If the induced map
\[
f^*:  \check{H}^*(X' ; \PP ) \to \check{H}^*(X ; \PP )
\]
is an isomorphism, then the induced map
\[
f^* : \pi _*(\GLo A_\zeta ' )\otimes\PP \to \pi _*(\GLo A_{\zeta })\otimes\PP
\]
is an isomorphism.
\end{Thm}

\begin{proof} The fact that the map (1) is an isomorphism implies that the map of spectral sequences
is an  isomorphism at the $E^2$  level and \lq\lq the standard comparison theorem\rq\rq\, implies that the
map of spectral sequences
is an  isomorphism at the $E^\infty $ level, which implies that (2) is an isomorphism.
\end{proof}

\section{The Stabilization Process}

Suppose   $\zeta : E = P \times _G B  \to X $ is a  bundle of Banach algebras over $\FFF$.
  Let $G$ act trivially on $M_n(\FFF  )$.  Form a new bundle
\[
\zeta _n : P \times _G (B\otimes M_n(\FFF  )) \to X
\]
(cf. \cite{MS} page 32)  over $X$;
the space of sections of this bundle is denoted
\[
A_{\zeta _n}      = \Gamma (X, P \times _G (B\otimes M_n(\FFF )) )  \cong A_\zeta  \otimes M_n(\FFF  ) .
\]
    It is elementary, then, to see that $\GL _nA_\zeta \equiv \GLo (A_{\zeta _n})$ consists of sections
\[
s: X \to E\otimes M_n(\FFF )
\]
  with the property that $s(x) \in \GL _n(E_x)$ for all $x$ and that the path component of
 the identity consists of  those sections connected by
a path of such sections to the identity section.

For any Banach algebra $A$ over $\FFF $ , there is a natural \lq\lq upper left corner\rq\rq\,
 homomorphism $\GL _nA \to \GL _{n+1}A $. We may form the associated homotopy groups and calculate
$  \Dirlim _n \pi _*(\GL _nA)$.  Bott periodicity \cite{Bott} and its extension to Banach algebras by  Wood \cite{Wood} and Karoubi \cite{Kar}
implies that these groups are periodic of period $2$ (if $\FFF=\CC$) or
$8$ (if $\FFF=\RR$) and
 \[
 \Dirlim _n \pi _*(\GL _nA) \overset\cong\longrightarrow \KF_{*+1}(A) .
\]
Note that the left hand side is $\ZZ $-graded, with $\pi _*  = 0$ for
$* < 0$. The right hand side is typically seen as $\ZZ/8 $-graded in
the real case and  $\ZZ/2 $-graded in the complex case
in recent literature,
but it is better to think of $\KF_j(A)$  as defined for   $j \geq 0$,\footnote{J. F. Adams \cite{Adams} page 619
notes that the Adams operations do not commute with periodicity and for that reason warns the reader \lq\lq We shall be most careful not to identify $K_\CC ^{-n-2}(X,Y)$  with
$K_\CC ^{-n}(X,Y)$ or $K_\RR ^{-n-8}(X,Y)$  with
$K_\RR ^{-n}(X,Y)$. We therefore regard $K_\Lambda (X,Y)$ as graded over $\ZZ$, not over $\ZZ _2 $ or $\ZZ _8$.\rq\rq }
by the results of Bott, Wood, and Karoubi.

We want to study the passage to the limit in the last isomorphism.
  In general the matter can be  complicated. For example, in the case $A = \CC $ the
groups $\pi _*(\GL _n\CC )$ have  vast amounts of torsion, but the limit group   $\Dirlim _n \pi _*(\GL _n\CC) \cong K_{*+1}(\CC) $
is simply $\ZZ $ or $0$ depending upon the parity of $*$.

Fix some standard  bundle $\zeta : E = P \times _G B  \to X$ of    $B$-algebras over $\FFF $  and associated
Banach algebra $A_\zeta $.
 Following the procedure outlined
above,  produce the sequence of Banach algebras $A_{\zeta _n}$.
Then there is a  spectral sequence
 \[
 E^2_{-p,q}(A_\zeta)   \cong \check{H}^p(X ; \pi _q(\GLo B)\otimes \PP )   \Longrightarrow    \pi _*(\GLo A_\zeta )\otimes \PP
 \]
 and, for each $n$, a spectral sequence
 \[
 E^2_{-p,q}(A_{\zeta_n})   \cong \check{H}^p(X ; \pi _q(\GL _nB)\otimes \PP )   \Longrightarrow  \pi _*(\GLo A_{\zeta_n} )\otimes \PP  .
 \]
The natural maps $\GLo B \to \GL _nB  \to \GL _{n+1}B$ induce morphisms of spectral sequences
\[
E^r_{*,*}(A_\zeta)     \longrightarrow
E^r_{*,*}(A_{\zeta_n})  \longrightarrow
E^r_{*,*}(A_{\zeta_{n+1}}) .
\]
The direct limit of spectral sequences is again a spectral sequence, since exactness is preserved under filtered direct
limits, and hence   form yet another spectral sequence
 \[
E^r_{*,*}(A_{\zeta _\infty }     )     \equiv    \Dirlim _n  E^r_{*,*}(A_{\zeta_n})  .
\]
(This is an abuse of notation since there is no bundle $\zeta _\infty$.)
The $E^2$ term of this spectral sequence is readily identified:
\[
E^2_{-p,q}(A_{\zeta _\infty }     )   \equiv    \Dirlim _n  E^2_{-p,q}(A_{\zeta_n}) = \Dirlim _n   \check{H}^p(X ; \pi _q(\GL _nB)\otimes \PP )  \cong
\]
\[
\cong \check{H}^p(X ; \Dirlim _n  \pi _q(\GL _nB)\otimes \PP )
\cong \check{H}^p(X ; \KF_{q+1}(B)\otimes\PP  ).
\]
Similarly, the spectral sequence converges to
\[
\Dirlim _n  \pi_*(\GL _nA_\zeta )\otimes\PP \cong \KF_{*+1}(A_\zeta )\otimes\PP .
\]
This completes the proof of Theorem \ref{T:theoremB} .
\vglue .2in

 Let $B$ be a    Banach algebra over $\FFF$. Recall   that   $B$ is {\emph{Bott-stable}} if the natural map
 \[
 \pi_*(\GLo B) \to \KF_{*+1}(B)
 \]
   is an isomorphism for every $* > 0$.   We note in the Introduction that many important
    operator algebras have this property.  Similarly,  call $B$ {\emph{rationally Bott-stable}} if the natural map
 \[
 \pi_*(\GLo B)\otimes \QQ  \to \KF_{*+1}(B)\otimes\QQ
 \]
   is an isomorphism.

Here is the proof of Theorem \ref{T:theoremC}.

\begin{proof}
We prove the first statement, the second being similar.
The natural map $\GLo B \to \GL _nB$  induces  a morphism of spectral sequences $\Phi _*$
from the spectral sequence
 \[
 E^2_{-p,q}(A_\zeta)   \cong \check{H}^p(X ; \pi _q(\GLo B)  )   \Longrightarrow    \pi _*(\GLo A_\zeta )
 \]
to the spectral sequence
 \[
 E^2_{-p,q}(A_{\zeta_\infty}  )   \cong \check{H}^p(X ; \KF _{q+1}(B)  )   \Longrightarrow    \KF _{*+1}(A_\zeta ) .
 \]
At the $E^2$ level the map $\Phi _2 $ is an isomorphism, because $B$ is Bott-stable. The standard comparison theorem implies
that  $\Phi _n$ is an isomorphism for each $n$ and hence the map
\[
\Phi _\infty :
\pi _*(\GLo A_\zeta )  \longrightarrow
\KF _{*+1}(A_\zeta )
\]
is an isomorphism. This is just the statement that $A_\zeta $ is Bott-stable.
\end{proof}

Note that the result does not imply anything about  differentials in the spectral sequence; there is no
reason to assume or conclude that $E^2 = E^\infty $.

\section{Collapse Theorems}

Recall that
an {\it $H$-space structure} on a based space $(X, \ast )$ is a map
$m\:X\times X\to X$ whose restriction to $X\times \ast$ and $\ast \times X$
is homotopic to the identity as based maps.
If an $H$-space structure on $X$ is understood, we call $X$ an {\it $H$-space.}
One says that $X$ is {\it homotopy associative} if the maps
$m \circ (m\times \text{id})$ and $m  \circ (\text{id}\times m)$ are homotopic.
A {\it homotopy inverse} for $X$ is a map $\iota\: X \to X$ such that the composites
$m \circ (\iota \times \text{id})$ and $m \circ (\text{id} \times \iota)$ are homotopic
to the identity.

\begin{Def} An $H$-space $X$  is  {\it group-like} if its multiplication is
  homotopy associative
and  has a homotopy inverse.
\end{Def}

 If $X$
is a group-like $H$-space  then the set of path components $\pi_0(X)$ acquires a group structure.

\begin{Thm}
Suppose that $X$ is a finite dimensional compact metric space, and suppose given a fibre bundle $\zeta : E = P \times _G B \to X $ of Banach algebras over $\CC $.
   Assume further that the
 classifying space of $\GLo B$ has the rational homotopy type of a group-like $H$-space. (This condition holds, for instance,
 if $\GLo B$ has rationally periodic homotopy groups.) Then
 \begin{enumerate}
 \item
 $\GLo A_\zeta $ is rationally $H$-equivalent to the function space  $F(X, \GLo B)$.
 \item   There is an isomorphism
  \[
 \pi_*(\GLo A_\zeta )\otimes\QQ \cong \check{H}^*(X ; \QQ ) \otimes \pi _*(\GLo B)\otimes\QQ  .
\]
 \item
 The spectral sequence of
  Theorem \ref{T:theoremA}  has $E^2 = E^\infty $.
\end{enumerate}
\end{Thm}

Note that an element $h\otimes x  \in \check{H}^*(X ; \QQ ) \otimes \pi _*(\GLo B)\otimes\QQ $ has total degree $|x| - |h|$ (corresponding
to the fact that the spectral sequence lies in the second quadrant). Our convention is to discard all terms of non-positive degree.

\begin{proof}  Part (1) of the Theorem is established in  (\cite{KSS}, page 264, Addendum E).   Then
the results of \cite{LPSS}      on the rational homotopy of function spaces of this type imply Part (2).
This of course implies that the spectral sequence must collapse.
\end{proof}

The Theorem immediately implies Corollary \ref{collapsecorollary} of the Introduction.

\begin{Rem} Our results can then be contrasted with the various results of Dadarlat (cf. \cite {Dar2009}, \cite{Dar2010}) who computes
$ \pi_*(\GLo A_\zeta )$   in terms of $K$-theory under various stringent assumptions on $B$.
\end{Rem}

\begin{Rem}\label{differentials}  In general there are many sources of differentials in the spectral sequence.  If the bundle is trivial then the problem reduces to
 the  homotopy of function spaces, and this is a complicated situation even rationally. Smith's paper \cite{Smith}  gives many examples of
non-trivial differentials in that case. In the infinite-dimensional 
continuous-trace situation explored by Atiyah and Segal \cite{AS},
they produce a plethora of differentials in the spectral sequence after passing to de Rham cohomology, and so (in contrast to the
classical Atiyah-Hirzebruch spectral sequence) this spectral sequence has differentials that do not vanish mod torsion.
\end{Rem}




\end{document}